\documentclass[12pt]{amsart}
\usepackage{enumitem}
\begin{document}

\newtheorem{thm}{Theorem}
\newtheorem{lem}{Lemma}[section]
\newtheorem{cor}[lem]{Corollary}
\theoremstyle{definition}
\newtheorem*{defn}{Definition}
\newtheorem*{ack}{Acknowledgements}
\theoremstyle{remark}
\newtheorem*{remarks}{Remarks}
\newcommand{\Z}{\mathbb{Z}}
\newcommand{\D}{\mathbb{D}}
\newcommand{\R}{\mathbb{R}}
\newcommand{\N}{\mathbb{N}}
\newcommand{\C}{\mathbb{C}}
\newcommand{\capac}{\operatorname{cap}}
\newcommand{\eqn}{\begin{equation}}
\newcommand{\eqnend}{\end{equation}}

\numberwithin{equation}{section}

\title[Slow escaping points of quasiregular maps]{Slow escaping points of quasiregular mappings}


\author{Daniel A. Nicks}
\address{School of Mathematical Sciences, University of Nottingham, Nottingham NG7 2RD, UK}
\email{Dan.Nicks@nottingham.ac.uk}
%\date{\today}
\thanks{The author was supported by Engineering and Physical Sciences Research Council grant EP/L019841/1.}

\begin{abstract}
This article concerns the iteration of quasiregular mappings on $\R^d$ and entire functions on $\C$. It is shown that there are always points at which the iterates of a quasiregular map tend to infinity at a controlled rate. Moreover, an asymptotic rate of escape result is proved that is new even for transcendental entire functions.

Let $f\colon\R^d\to\R^d$ be quasiregular of transcendental type. Using novel methods of proof, we generalise results of Rippon and Stallard in complex dynamics to show that the Julia set of $f$ contains points at which the iterates $f^n$ tend to infinity arbitrarily slowly. We also prove that, for any large $R$, there is a point $x$ with modulus approximately $R$ such that the growth of $|f^n(x)|$ is asymptotic to the iterated maximum modulus $M^n(R,f)$.
\end{abstract}
\maketitle

\section{Introduction}

In complex dynamics, a central object of study is the \emph{escaping set}
\[I(f)=\{ z\in\C : f^n(z)\to\infty\} \]
 of a transcendental entire function~$f$. There is much interest in both the structure of escaping sets \cite{Rempe11, RSFast, RRRS, Sixsmith11} and also various subsets consisting of points at which the iterates tend to infinity particularly quickly or particularly slowly \cite{BP, RSZip, RSSlow, SixsmithSlow}. Quasiregular mappings are a natural higher-dimensional generalisation of analytic functions, and there is a growing literature exploring the analogies between complex dynamics and quasiregular dynamics; see for example \cite{Bergw2013, BDF, BN, HMM, NS}. This paper investigates how the rate of escape of the iterates of a quasiregular mapping can be controlled.

For a transcendental entire function $f$, Eremenko \cite{Eremenko89} proved that the escaping set is always non-empty and in fact always meets the Julia set~$J(f)$. The fast escaping set~$A(f)$ was introduced by Bergweiler and Hinkkanen~\cite{BH} and has been studied in detail by Rippon and Stallard~\cite{RS05, RSFast, RS13, RS14}. This subset of $I(f)$ is often described as containing those points that escape to infinity `as fast as possible' under iteration of the entire function~$f$ and it is again known that ${A(f)\cap J(f) \ne\emptyset}$. Motivated partly by the question whether there could be a transcendental entire function for which all escaping points were actually fast escaping, Rippon and Stallard proved in \cite{RSSlow} that, for any transcendental meromorphic function~$f$, there are points in the Julia set at which the iterates $f^n$ tend to infinity arbitrarily slowly. Theorems~\ref{thm1} and~\ref{thm3} below generalise some of the main results of \cite{RSSlow} to the quasiregular setting.

The definition of a quasiregular mapping $f\colon\R^d\to\R^d$ is given in Section~\ref{sect2}.  If $|f(x)|\to\infty$ as $|x|\to\infty$, then the mapping is said to be of \emph{polynomial type}. Otherwise, $f$ has an essential singularity at infinity and is called \emph{transcendental type}. It was shown in \cite{BFLM} that a quasiregular mapping~$f$ of transcendental type always has a non-empty escaping set,
\[ I(f) = \{ x\in\R^d : f^n(x)\to\infty \mbox{ as } n\to\infty\} \ne \emptyset. \]
Indeed, it was later proved that the fast escaping set~$A(f)$ is also non-empty~\cite{BDF}. Our first result below is a quasiregular counterpart to \cite[Theorem~1]{RSSlow}: we will see that there are points in the Julia set that escape to infinity arbitrarily slowly. In this quasiregular context, we follow \cite{Bergw2013, BN} in defining the \emph{Julia set} $J(f)$ to be the set of all $x\in\R^d$ such that the complement of $\bigcup_{n=1}^\infty f^n(U)$ has conformal capacity zero for every neighbourhood $U$ of~$x$. (Note that sets of conformal capacity zero are small; for example, they have Hausdorff dimension zero. See~\cite{Rickman} for the definition.) For a transcendental entire function, this is equivalent to the classical non-normality definition of the Julia set; see \cite{BN}.

\begin{thm}\label{thm1}
Let $f\colon\R^d\to\R^d$ be a quasiregular map of transcendental type. Then, for any positive sequence $a_n\to\infty$, there exist $\zeta\in J(f)$ and $N_1\in\N$ such that $|f^n(\zeta)|\to\infty$ as $n\to\infty$, but also $|f^n(\zeta)|\le a_n$ for $n\ge N_1$.
\end{thm}

To state our next result, we recall that the \emph{maximum modulus} is given by $M(r,f)=\max_{|x|=r} |f(x)|$ and we denote the \emph{iterated maximum modulus} by
\[ M^1(r,f) = M(r,f) \quad \mbox{ and } \quad M^n(r,f) = M(M^{n-1}(r,f),f) \mbox{ for } n\ge2. \]
Observe that if $|x|\le r$, then $|f^n(x)|\le M(r,f^n) \le M^n(r,f)$ by the maximum principle. For quasiregular maps of transcendental type, we show that we can always find points for which the rate of escape is asymptotic to an iterated maximum modulus. This result is new even for transcendental entire functions in the plane.

\begin{thm}\label{thm2}
Let $f\colon\R^d\to\R^d$ be a quasiregular map of transcendental type and let $c>1$. Then there exists $r_1$ such that, for every $R\ge r_1$, there is $x\in I(f)$ satisfying $R/c \le |x| \le cR$ and
\[ \lim_{n\to\infty} \frac{|f^n(x)|}{M^n(R,f)} = 1. \]
\end{thm}

In the special case of transcendental entire functions, the escaping point found in Theorem~\ref{thm2} must actually have modulus at least~$R$. This follows from \cite[Lemma~2.2]{RS09}, for example. We mention also that, for a transcendental entire function $f$, Sixsmith \cite{Sixsmith} has studied the set of escaping points $z\in\C$ for which some iterate $y=f^N(z)$ satisfies $|f^n(y)|=M^n(|y|,f)$ for all $n\ge 1$, although examples show that this set may be empty.

It is already known that if $f$ is quasiregular of transcendental type, $c>1$ and $R$ is large, then there exists a point $x$ such that $R\le |x| \le cR$ and $|f^n(x)|\ge M^n(R,f)$ for $n\ge1$. This follows from results in \cite{BDF}. In the entire case, Eremenko's original method for finding escaping points can be used to obtain such points; see the proof of \cite[Lemma~2]{BH} and also \cite[Theorem~2.5]{RSFast}. In each of these existing constructions, however, the iterates $f^n(x)$ may tend to infinity much faster than $M^n(R,f)$. Hence, Theorem~\ref{thm2} may be viewed as a deceleration of these previous results. Nonetheless, the points obtained in Theorem~\ref{thm2} are still members of the fast escaping set $A(f)$.

One can ask whether it is possible in general to find points satisfying any given upper and lower bounds on the rate of escape. Our final result gives necessary and sufficient conditions under which a two-sided version of Theorem~\ref{thm1} holds.

\begin{thm}\label{thm3}
Let $f\colon\R^d\to\R^d$ be a quasiregular map of transcendental type. Then the following are equivalent:
\begin{itemize}
	\item[(a)] for every positive sequence $a_n\to\infty$ satisfying $a_{n+1}=O(M(a_n,f))$, there exist $\zeta\in\R^d$ and $C>1$ such that, for all $n\in\N$,
	\[ a_n \le |f^n(\zeta)| \le Ca_n; \]
	\item[(b)] there exist positive constants $L$ and $c$ and a sequence of points $(x_n)$ tending to infinity such that, for all $n\in\N$,
	\[ \frac{|x_{n+1}|}{|x_n|} \le L \quad \mbox{ and } \quad |f(x_n)|\le c; \]
	\item[(c)] there exist constants $L'>1$ and $0<s<1$ and a sequence of points $(x'_n)$ tending to infinity such that, for all $n\in\N$,
	\[ \frac{|x'_{n+1}|}{|x'_n|} \le L' \quad \mbox{ and } \quad |f(x'_n)|\le M(|x'_n|,f)^s. \]
\end{itemize}
Moreover, when the above all hold, the point $\zeta$ appearing in \textup{(a)} may be chosen to lie in $J(f)$.
\end{thm}

This result was proved for transcendental meromorphic functions with finitely many poles in \cite[Theorem~2]{RSSlow}. We repeat the remark made there that it is clear that some restriction on the sequence $(a_n)$ is needed, such as $a_{n+1}=O(M(a_n,f))$ as $n\to\infty$.

The proofs of Theorems \ref{thm1} and \ref{thm3} given here are similar to the corresponding proofs in \cite{RSSlow}, in that they use a number of different covering results combined with a `holding-up' technique to produce slowly-escaping points. A significant difference is that one key covering result (Lemma~5) in \cite{RSSlow} makes use of a convexity property  of $\log M(r,f)$ that does not hold for quasiregular mappings (cf.~\cite[Theorem~1.8]{BDF}). Instead, here we consider as a separate case those functions that are large outside certain small `pits'; see Section~\ref{sect:pits}. Other covering results will be established by normal families methods applied to suitable families of rescalings.

\begin{remarks}
\mbox{}
\begin{itemize}[leftmargin=*]
\item The escaping points in Theorem~\ref{thm2} cannot necessarily be chosen to lie within the Julia set. For example, it is shown in \cite{Zheng} that, for any transcendental entire function $f$ with a multiply-connected Fatou component, and any ${c>1}$, there are arbitrarily large $R$ such that $J(f)$ is disjoint from the annulus ${\{z\in\C : R/c \le |z| \le cR\}}$.
\item For quasiregular mappings of polynomial type, there is much less variation in the rates at which different orbits escape to infinity; see \cite[Theorem~1.4]{FN4}.
\end{itemize}
\end{remarks}

\begin{ack}
The author thanks Phil Rippon and Gwyneth Stallard for helpful discussions about their paper \cite{RSSlow} and the possibility of extending slow escape results to the quasiregular setting.
\end{ack}

\section{Preliminary results}\label{sect2}

We will begin by recalling the definition of a quasiregular mapping and stating just those properties that we shall need later. We refer to \cite{Rickman} for a more detailed introduction to the theory of quasiregular maps.

Let $d\ge 2$ and let $U \subset \R^d$ be a domain. For $1\le p < \infty$,  the \emph{Sobolev space} $W^{1}_{p, \textrm{loc}}(U)$
consists of those functions $f\colon U \rightarrow \R^d$ for which all first order weak partial derivatives
exist and are locally in $L^p$.
A continuous map $f\in W^{1}_{d, \textrm{loc}}(U)$ is called \emph{quasiregular} if there exists $K \ge1$ such that
\[
|Df(x)|^d \le K J_{f}(x)
\]
for almost every $x\in U$. Here $|Df(x)|$ is the operator norm of the derivative $Df(x)$, and $J_{f}(x)$ denotes the Jacobian determinant. The smallest constant $K$ for which this holds is called the \emph{outer dilatation} $K_{O}(f)$. If $f$ is quasiregular, then there also exists $K' \ge 1$ such that
\[
J_{f}(x) \le K' \inf_{|h|=1} |Df(x) (h)|^d
\]
for almost every $x\in U$, and the smallest $K'$ for which this holds is called the \emph{inner dilatation} $K_{I}(f)$. The \emph{dilatation} $K(f)$ of $f$ is the larger of $K_{O}(f)$ and $K_{I}(f)$, and we say that $f$ is $K$-quasiregular if $K(f)\le K$. Any function obtained by composing or iterating quasiregular mappings will itself be quasiregular. It turns out that all non-constant quasiregular maps are discrete open mappings. Note that entire functions on the complex plane are quasiregular with dilatation $K(f)=1$.

Rickman established the following quasiregular analogue of Picard's theorem.

\begin{lem}[\cite{Rickman80}]\label{lem:Rickman}
For every $d\ge 2$ and $K\ge 1$, there is an integer ${q=q(d,K)}$, called \emph{Rickman's constant}, with the following properties. If ${a_1,\ldots, a_q\in\R^d}$ are distinct, then any $K$-quasiregular mapping ${f\colon\R^d\to\R^d\setminus\{a_1,\ldots,a_q\}}$ is constant.  Moreover, for any $K$-quasiregular mapping $f\colon \R^d\to\R^d$ of transcendental type, $f^{-1}(\{a_1,\ldots,a_q\})$ is infinite.
\end{lem}

An immediate consequence of this result is that, for any $K$-quasiregular mapping~$f$ of transcendental type, the set
\[ \mathcal{E}(f)=\{x\in\R^d : f^{-1}(x) \mbox{ is finite} \}  \]
contains at most $q-1$ points.

The definition of quasiregularity can be extended to maps between Riemannian manifolds. In particular, if $\overline{\R^d}=\R^d\cup\{\infty\}$ is equipped with the spherical metric and $U\subset\overline{\R^d}$ is a domain, then a continuous map ${f\colon U\to\overline{\R^d}}$ is called \emph{quasimeromorphic} if $f^{-1}(\infty)$ is discrete and $f$ is quasiregular on ${U\setminus(f^{-1}(\infty)\cup\{\infty\})}$. Miniowitz used Rickman's theorem to prove the following analogue of Montel's theorem. Here we write $\chi(x,y)$ for the chordal distance between two points $x,y\in\overline{\R^d}$.

\begin{lem}[{\cite[Theorem 5]{Miniowitz}}]\label{lem:Minio}
Let $\mathcal{F}$ be a family of $K$-quasimeromorphic functions on a domain $U\subset \R^d$ and let $q=q(d,K)$ be Rickman's constant. If, for some $\varepsilon>0$, each function $f\in\mathcal{F}$ omits $q+1$ values $a_1(f),\ldots, a_{q+1}(f) \in\overline{\R^d}$ such that $\chi(a_i(f), a_j(f))\ge\varepsilon$ for $i\ne j$, then $\mathcal{F}$ is a normal family on $U$.
\end{lem}

We will make use of the following generalisation to the quasiregular setting of two well-known results about the rate of growth of transcendental entire functions.

\begin{lem}[{\cite[Lemmas 3.3 and 3.4]{B}}]\label{lem:M(Ar)}
Let $f\colon \R^d\to\R^d$ be a quasiregular map of transcendental type and let $A>1$. Then
\[ \lim_{r\to\infty} \frac{M(Ar,f)}{M(r,f)} = \infty \quad \mbox{ and } \quad \lim_{r\to\infty} \frac{\log M(r,f)}{\log r} = \infty. \]
\end{lem}

We note the following facts about the Julia set of a quasiregular mapping, as defined in the introduction.

\begin{lem}[\cite{BN}]\label{lem:Julia1}
Let $f\colon\R^d\to\R^d$ be a quasiregular mapping of transcendental type. Then the Julia set $J(f)$ is non-empty. In fact, $J(f)$ is infinite. The Julia set is completely invariant; that is, $x\in J(f)$ if and only if $f(x)\in J(f)$.
\end{lem}

The next lemma tells us that a collection of domains must meet the Julia set if the image of each domain in the collection contains many of the others. We denote the closure of a set $E$ by $\overline{E}$.

\begin{lem}\label{lem:Julia2}
Let $f\colon\R^d\to\R^d$ be a $K$-quasiregular map of transcendental type and let $p\in\N$ be such that $p>K_I(f)+q$, where $q=q(d,K)$ is Rickman's constant. Suppose that $U_1,\ldots,U_p$ are pairwise disjoint bounded domains such that, for each $j\in\{1,\ldots,p\}$,
\[ f(U_j) \supset U_i \quad \mbox{ for at least } p-q \mbox{ values of } i\in\{1,\ldots,p\}. \]
Then $\overline{U_j}\cap J(f)\ne\emptyset$ for every $j\in\{1,\ldots,p\}$.
\end{lem}

\begin{proof}
This can be deduced from \cite[Theorem~3.2]{Bergw2013} and the definition of the Julia set $J(f)$. The argument is exactly the same as that used on page~161 of~\cite{BN} in the ``Proof of Theorem~1.1 for functions without the pits effect'' (simply replace the $g_m$ and $N$ referred to there by $f$ and $p$ respectively).
\end{proof}

The final result in this section is essentially \cite[Lemma~1]{RSSlow}.

\begin{lem}\label{lem:En}
Let $f\colon\R^d\to\R^d$ be a continuous function and let $(E_n)$ be a sequence of non-empty bounded sets in $\R^d$ such that
\[ f(E_n) \supset E_{n+1} \quad \mbox{ for } n\ge 0. \]
Then there exists $\zeta\in\R^d$ such that $f^n(\zeta)\in\overline{E_n}$ for $n\ge0$.

If $f$ is also quasiregular of transcendental type and there is a subsequence $(E_{n_k})$ such that $\overline{E_{n_k}}\cap J(f)\ne\emptyset$, then the point $\zeta$ may be chosen to lie in $J(f)$.
\end{lem}

\begin{proof}
The continuity of $f$ and compactness of $\overline{E_n}$ imply that $f(\overline{E_n}) \supset \overline{E_{n+1}}$ for $n\ge 0$. Thus the sets
\[ F_n = \{x\in \overline{E_0} : f(x)\in\overline{E_1},\ldots, f^n(x)\in\overline{E_n} \} \]
are non-empty, compact and form a decreasing nested sequence. Hence the intersection $F=\bigcap_{n\ge0} F_n$ is non-empty and any $\zeta\in F$ satisfies $f^n(\zeta)\in\overline{E_n}$ for $n\ge0$.

To prove the second part of the lemma, we use the fact that the Julia set $J(f)$ is completely invariant under $f$. Using this fact, we deduce first that $\overline{E_{n}}\cap J(f)\ne\emptyset$ for every $n\ge 0$. The complete invariance of $J(f)$ then allows us to apply the first part of the lemma to the sets $\overline{E_{n}}\cap J(f)$, which yields $\zeta\in J(f)$ as required.
\end{proof}

\section{Proof of Theorem~\ref{thm1}}

\subsection{A holding-up lemma}\label{sect:holdup}

Theorem~\ref{thm1} will be proved by using Lemma~\ref{lem:En} to produce a point in $I(f)\cap J(f)$. In order to control the speed of escape to infinity, a `holding-up' technique similar to that of \cite[Lemma~6]{RSSlow} will be used to construct an escaping point whose forward iterates are repeatedly made to wait in some set for a prescribed number of steps. The holding-up part of the argument is contained in the next lemma. The actual construction of sets with the required covering properties under $f$ will be the subject of the remainder of the proof.

\begin{lem}\label{lem:holdup}
Let $f$ be an entire quasiregular map of transcendental type and, for $\nu\in\N$ and $j\in\{1,\ldots,p\}$, let $A_\nu^{(j)}$ be non-empty bounded sets such that the unions $A_{\nu}=\bigcup_{j=1}^p A_\nu^{(j)}$ tend to infinity in the sense that
\eqn
\lim_{\nu\to\infty} \inf \{|x|:x\in A_\nu\} = \infty. \label{Ha}
\eqnend
Suppose further that
\begin{multline}\label{Hb}
	\mbox{for each } \nu\in\N \mbox{ and } j\in\{1\ldots,p\}, \\ \mbox{ we have }  f(A_\nu^{(j)}) \supset A_{\nu+1}^{(i)}  \mbox{ for some } i\in\{1,\ldots,p\},
\end{multline}
and that there exists a strictly increasing sequence of integers $(\nu_k)$ such that
\begin{multline}\label{Hc}
	\mbox{for each } k\in\N \mbox{ and } j\in\{1\ldots,p\}, \\ \mbox{ we have }  f(A_{\nu_k}^{(j)}) \supset A_{\nu_k}^{(i)}  \mbox{ for some } i\in\{1,\ldots,p\}
\end{multline}
and that, for every $j\in\{1,\ldots,p\}$,
\eqn
\overline{A_{\nu_k}^{(j)}} \cap J(f) \ne \emptyset. \label{Hd}
\eqnend
Then, given any positive sequence $a_n\to\infty$, there exist $\zeta\in J(f)$ and $N_1\in\N$ such that $|f^n(\zeta)|\to\infty$ as $n\to\infty$, but also $|f^n(\zeta)|\le a_n$ for $n\ge N_1$.
\end{lem}

\begin{proof}
Define an increasing real sequence $(\rho_\nu)$ by
\eqn
 \rho_\nu = \sup \left\{|x| : x\in A_1 \cup \ldots \cup A_\nu \right\}. \label{defn_rho}
\eqnend
Now choose a strictly increasing sequence of positive integers $N_k$ such that, for all $n\ge N_k$, we have
\[ \rho_{\nu_k} \le a_n. \]
This is possible because the sequence $(a_n)$ tends to infinity. Next, for every integer $n\ge N_1$ we inductively define an integer $\mu(n)$ as follows. Set $\mu(N_1)=\nu_1$. Then, for every $n\ge N_1$, we have $N_k\le n < N_{k+1}$ for some $k$ and we define
\[ \mu(n+1) = \begin{cases} \mu(n)+1, & \mbox{if } \mu(n)<\nu_k; \\
\mu(n), & \mbox{otherwise}. \end{cases} \]
Note that the latter case occurs if and only if $\mu(n)=\nu_k$. Typically, the sequence $(\mu(n))_{n\ge N_1}$ looks something like
\[ \nu_1, \ldots, \nu_1, \nu_1+1, \nu_1+2, \ldots, \nu_2, \nu_2, \ldots, \nu_2, \nu_2+1, \ldots, \nu_3, \nu_3, \ldots; \]
that is, it counts up through the integers but may pause for a finite number of steps at each value $\nu_k$. In particular, $\mu(n)\to\infty$ as $n\to\infty$.

We now claim that
\eqn
\rho_{\mu(n)} \le a_n \quad \mbox{ for all } n\ge N_1; \label{an>p}
\eqnend
that is, we have held up the sequence $\rho_{\mu(n)}$ so that it grows more slowly than the $a_n$. To prove the claim, let $n\ge N_1$ and find $k$ such that $N_k\le n < N_{k+1}$. Then $\rho_{\nu_k} \le a_n$ by the definition of $N_k$. The way $\mu(n)$ was chosen means that $\mu(n)\le\nu_k$ and so, as $(\rho_\nu)$ is increasing, we have that $\rho_{\mu(n)} \le \rho_{\nu_k} \le a_n$, which proves \eqref{an>p}.

From hypotheses \eqref{Hb} and \eqref{Hc} it follows that, for all $n\ge N_1$ and ${j\in\{1,\ldots,p\}}$, we have
\[  f\left(A_{\mu(n)}^{(j)}\right) \supset A_{\mu(n+1)}^{(i)} \mbox{ for some } i\in\{1,\ldots,p\}. \]
Here we have used the fact that either $\mu(n+1)=\mu(n)+1$ or else $\mu(n+1)=\mu(n)=\nu_k$ for some $k$. Therefore, for all $n\ge N_1$, we can choose $E_n$ to be one of the sets $A_{\mu(n)}^{(1)},\ldots,A_{\mu(n)}^{(p)}$ in such a way that $f(E_n)\supset E_{n+1}$. Combining this with \eqref{Hd}, an application of Lemma~\ref{lem:En} now yields a point $\zeta_{N_1}\in J(f)$ such that $f^{n-N_1}(\zeta_{N_1})\in\overline{E_n}$ for $n\ge N_1$. Without loss of generality, we may assume that $\zeta_{N_1}\notin\mathcal{E}(f)$ (else increase $N_1$) and so applying Lemma~\ref{lem:Rickman} finitely many times we can find $\zeta$ such that $f^{N_1}(\zeta)=\zeta_{N_1}$. Since the Julia set is completely invariant, we have $\zeta\in J(f)$. Moreover, we have that $f^n(\zeta)\in\overline{E_n}\subset \overline{A_{\mu(n)}}$ for $n\ge N_1$ and therefore
\[ |f^n(\zeta)| \le \rho_{\mu(n)} \le a_n \]
by \eqref{defn_rho} and \eqref{an>p}, but also $|f^n(\zeta)|\to\infty$ as $n\to\infty$ by \eqref{Ha}.
\end{proof}

\subsection{The pits effect}\label{sect:pits}

In view of Lemma~\ref{lem:holdup}, to prove Theorem~\ref{thm1} it will suffice to construct sets with the properties \eqref{Ha}--\eqref{Hd}. This will be carried out in two distinct cases, depending on whether or not the transcendental type quasiregular mapping in question has the ``pits effect''. The same division of cases is central to the proof of the non-emptiness of the Julia set in \cite{BN} and to other results proved there. The definition, exactly as given in \cite[Definition~1.2]{BN}, is as follows.

\begin{defn}
A quasiregular map $f\colon \R^d\to \R^d$ of transcendental type is said to have the \emph{pits effect}
if there exists $N\in \N$ such that, for all $c>1$ and all $\varepsilon>0$, there exists $R_0$ such that if $R>R_0$, then
\[ \{x\in \R^d:R\le |x|\le cR,\ |f(x)|\le 1 \}\]
can be covered by $N$ balls of radius $\varepsilon R$.
\end{defn}

The following result from \cite{BN} uses a quasiregular version of Harnack's inequality (see Lemma~\ref{lem:Harnack} below) to show that an equivalent definition of the pits effect can be given. Namely, the condition that $|f(x)|\le 1$ only within certain small `pits' can be replaced by a condition that, at least initially, appears stronger.

\begin{lem}[{\cite[Theorem~8.1]{BN}}]\label{lem:Harnack-pits}
Let $f\colon \R^d\to \R^d$ be a quasiregular map of transcendental type that
has the pits effect.
Then there exists $N\in \N$ such that for all $\alpha>1$,
all $c>1$ and
all $\varepsilon>0$ there exists $R_0$ such that if $R>R_0$, then
\[ \{x\in \R^d:R\le |x|\le cR,\ |f(x)|\le R^\alpha \} \]
can be covered by $N$ balls of radius $\varepsilon R$.
\end{lem}

\subsection{Proof of Theorem~\ref{thm1} for functions with the pits effect}

In this subsection we shall assume that $f$ is a transcendental type quasiregular mapping that has the pits effect. We aim to prove Theorem~\ref{thm1} for such~$f$ by using Lemma~\ref{lem:Harnack-pits} to satisfy the conditions for holding-up via Lemma~\ref{lem:holdup}.

Informally, the next lemma provides sets $V_k$ in which an orbit may be held up, while the following Lemma~\ref{lem:pits2} will provide the means by which an orbit can be advanced from $V_k$ to $V_{k+1}$.

\begin{lem}\label{lem:pits1}
There exist $0<\delta\le \frac12$ and a sequence of points $(x_k)$ tending to infinity such that the moduli $R_k=|x_k|$ are increasing and the balls $V_k=B(x_k,\delta R_k)$ satisfy
\eqn
f(V_k) \supset B(0,2R_k) \supset V_k. \label{Vk}
\eqnend
\end{lem}

\begin{proof}
Let $(x_k)$ be a sequence tending to infinity such that $|f(x_k)|\le 1$ and let $N\in\N$ be as in Lemma~\ref{lem:Harnack-pits}. We can assume that $R_k=|x_k|$ is increasing.

We claim that there exists $\delta\in\{\frac{1}{4N},\ldots,\frac{N+1}{4N}\}$ such that, after passing to a subsequence of $(x_k)$, we have
\eqn
\inf_{x\in \partial B(x_k,\delta R_k)} |f(x)| \ge 2R_k \label{inf|f|>2Rk} \quad \mbox{ for all } k.
\eqnend
Suppose not, then for all large $k$ and all $l\in\{1,\ldots,N+1\}$ there must be some point $y\in \partial B(x_k, lR_k/4N)$ such that $|f(y)|<2R_k$. In particular, this means that
\[ \{x\in\R^d: |x-x_k|\le R_k/2,\ |f(x)|\le 2R_k \} \]
cannot be covered by $N$ balls of radius $R_k/9N$. For large $k$, this set is contained in
\[ \{ x\in \R^d : R_k/2 \le |x| \le 3R_k/2,\ |f(x)|\le (R_k/2)^2 \} \]
and hence this contradicts Lemma~\ref{lem:Harnack-pits} with $\alpha=2$, $c=3$ and $\varepsilon=2/9N$. This proves the claim.

It remains to prove that \eqref{Vk} holds for the balls $V_k=B(x_k,\delta R_k)$. It is clear that $V_k\subset B(0,2R_k)$. We note that $\partial f(V_k) \subset f(\partial V_k) = f(\partial B(x_k,\delta R_k))$ because quasiregular mappings are open. Therefore it follows from \eqref{inf|f|>2Rk} that ${\partial f(V_k) \cap B(0,2R_k)=\emptyset}$. Since $|f(x_k)|\le 1$, we obtain the set inclusion \eqref{Vk}.
\end{proof}

The maximum modulus function $M(r,f)$ is continuous and increasing in $r$ and, by Lemma~\ref{lem:M(Ar)}, $M(r,f)>2r$ for all large $r$. We now assume that the values $R_k$ given by Lemma~\ref{lem:pits1} are so large that the real numbers $s_k$ defined by $M(s_k,f)=R_k$ are large enough that $M(s_k,f)>2s_k$. Note further that $s_k\to\infty$ as $k\to\infty$.

We denote spherical shells centred at the origin by
\[ A(s,t) = \{x\in\R^d : s<|x|<t \}. \]

\begin{lem}\label{lem:pits2}
For $k\in\N$ and $t\ge 2R_k$,
\[ f(A(s_k,t)) \supset A(s_k, 2t). \]
\end{lem}

\begin{proof}
As we are assuming that $f$ has the pits effect and $R_k$ is large, it follows from Lemma~\ref{lem:Harnack-pits} that, given any $t\ge 2R_k$, there exists $t'\in[\frac{9}{10}t,t]$ such that
\[ \inf_{|x|=t'} |f(x)|\ge 2t.\]
By an argument similar to that at the end of the previous proof, this shows that
\[ B(0,2t) \subset f(B(0,t')) \subset f(B(0,t)). \]
Now let $x\in A(s_k,2t)$. By the above, there is $y\in B(0,t)$ such that $f(y)=x$. Observe that either $y\in A(s_k,t)$ as required, or else $|y|\le s_k$. In the latter case, $|x|=|f(y)|\le M(s_k,f)= R_k$ and so by Lemma~\ref{lem:pits1} there is $z\in V_k$ such that $f(z)=x$. The proof is then completed by noting that
\[ V_k = B(x_k,\delta R_k) \subset A\!\left(\tfrac12 R_k, \tfrac32  R_k\right) \subset A(s_k, t). \qedhere \]
\end{proof}

Next we will choose a sequence of sets $A_\nu = A_\nu^{(1)}$ (we shall omit the superscript) that will satisfy the hypotheses of Lemma~\ref{lem:holdup} with $p=1$. In brief, the sets $V_k$ will form the subsequence $A_{\nu_k}$ at which `holding-up' can occur (as in \eqref{Hc}), while the rest of the $A_\nu$ will be chosen to be spherical shells like those appearing in Lemma~\ref{lem:pits2}. The details are as follows.

Begin by putting $\nu_1=1$ and inductively defining $\nu_{k+1}=\nu_k +L$, where $L$ is the least integer greater than $1$ such that $\frac32 R_{k+1}\le 2^LR_k$. Set $A_{\nu_k}=V_k$. The condition \eqref{Hc} is then satisfied by \eqref{Vk} of Lemma~\ref{lem:pits1}. Moreover, since $R_k$ is large, we can assume that $B(0,2R_k)\cap J(f)\ne\emptyset$ and therefore \eqref{Vk} and the complete invariance of $J(f)$ together imply that $V_k\cap J(f)\ne\emptyset$; in particular, condition \eqref{Hd} is fulfilled.

Next, for those integers $\nu$ in the range $\nu_k<\nu<\nu_{k+1}$, we set
\[ A_\nu = A(s_k, 2^{\nu-\nu_k}R_k).\]
Observe that the whole sequence $(A_\nu)$ tends to infinity in the sense of \eqref{Ha} because both $R_k\to\infty$ and $s_k\to\infty$ as $k\to\infty$. We verify that condition \eqref{Hb} is satisfied in three cases:
\begin{enumerate}
	\item When $\nu=\nu_k$ for some $k$, we have that
	\[ f(A_{\nu_k}) = f(V_k)  \supset B(0,2R_k) \supset A(s_k, 2R_k) = A_{\nu_k +1}, \]
	using \eqref{Vk} of Lemma~\ref{lem:pits1}.
	\item When $\nu_k<\nu<\nu_{k+1}-1$ for some $k$, we have that
	\[ f(A_\nu) = f(A(s_k, 2^{\nu-\nu_k}R_k)) \supset A(s_k, 2^{\nu+1-\nu_k}R_k) = A_{\nu+1}, \]
	by Lemma~\ref{lem:pits2}.
  \item When $\nu=\nu_{k+1}-1$ for some $k$, we have that
  \begin{align*}
  f(A_\nu)  = f(A(s_k, 2^{\nu_{k+1}-1-\nu_k}R_k)) & \supset A(s_k, 2^{\nu_{k+1}-\nu_k}R_k) \\
  & \supset A\!\left(s_k, \tfrac32 R_{k+1}\right),
  \end{align*}
  by using Lemma~\ref{lem:pits2} and recalling the definition of the subsequence~$(\nu_k)$. The remarks preceding Lemma~\ref{lem:pits2} give that $s_k<M(s_k,f)/2 = R_k/2$ and so, since $R_k\le R_{k+1}$, we now see that
  \[ f(A_\nu) \supset A\!\left(s_k, \tfrac32 R_{k+1}\right) \supset A\!\left(\tfrac12 R_{k+1}, \tfrac32 R_{k+1}\right) \supset V_{k+1} = A_{\nu+1}. \]
\end{enumerate}

Since the sequence of sets $(A_\nu)$ satisfies all the hypotheses of Lemma~\ref{lem:holdup}, an application of that result now completes the proof of Theorem~\ref{thm1} for functions having the pits effect.

\subsection{Proof of Theorem~\ref{thm1} for functions without the pits effect}\label{sect:no-pits}

The~holding-up technique will be used again to find slowly-escaping points for functions that do not have the pits effect. The structure of the argument is similar to that of the previous subsection: this time the `holding-up' will be provided by Lemma~\ref{lem:Uj} below, while the `marching forward' in between will be handled first by Lemma~\ref{lem:Qj}. In fact, the latter result will hold regardless of the presence of the pits effect and a modification of it will form the basis of the proof of Theorem~\ref{thm2} in Section~\ref{sect:thm2}.

We begin by defining domains $Q_j(r)\subset \R^d$ as follows. Let $q$ be a positive integer and fix $2q$ distinct unit vectors $v_1,\ldots,v_{2q}$ (that is, $v_j\in\R^d$ and $|v_j|=1$ for each $j$). Fix $\theta>0$ so small that, for $j\in\{1,\ldots,2q\}$, the truncated cones
\[ C_j = \left\{x\in\R^d : \frac14 < |x| < 2q+1, \ \frac{v_j\cdot x}{|x|}>\cos\theta \right\} \]
have pairwise disjoint closures, where $v\cdot x$ denotes the usual scalar product. Define, for $r>0$ and $j\in\{1,\ldots,2q\}$,
\[ Q_j(r) = A\!\left(jr, \left(j+\tfrac12\right)r\right) \cup rC_j. \]
Here and subsequently we use the notation $rE= \{rx:x\in E\}$. Note that $Q_j(r)=rQ_j(1)$. The following lemma exploits the fact that, although the sets $Q_j(1)$ are not disjoint, no point lies in more than two of these sets.

\begin{lem}\label{lem:Q-Minio}
Let $\mathcal{F}$ be a family of $K$-quasiregular functions on a domain $D\subset \R^d$ and let $q=q(d,K)$ be Rickman's constant. If every function in $\mathcal{F}$ omits a value in each of the sets $Q_1(1),\ldots,Q_{2q}(1)$, then $\mathcal{F}$ is normal.
\end{lem}

\begin{proof}
Take $0<\varepsilon\le\frac12$ such that, for $i\ne j$,
\[ \operatorname{dist}(C_i,C_j) = \inf\{|x-y| : x\in C_i, \, y\in C_j\} \ge \varepsilon. \]
Note also that, when $i\ne j$,
\[ \operatorname{dist}\left(A\!\left(i,i+\tfrac12\right), A\!\left(j,j+\tfrac12\right)\right) = |i-j|-\tfrac12 \ge \tfrac12. \]

Now consider any set $B=\{b_1,\ldots,b_{2q}\}$ such that $b_j\in Q_j(1)$ for each~$j$. Since each $b_j$ belongs to either $A\!\left(j,j+\tfrac12\right)$ or $C_j$, there must exist a subset $\{\beta_1,\ldots, \beta_q\}\subset B$ of cardinality $q$ with the property that $|\beta_i-\beta_j|\ge \varepsilon$ for $i\ne j$.

The normality of $\mathcal{F}$ may thus be deduced from Lemma~\ref{lem:Minio}, in view of the fact that $\mathcal{F}$ is a family of quasiregular functions and the sets $Q_j(1)$ are bounded away from infinity in the chordal metric.
\end{proof}

Next we apply Lemma~\ref{lem:Q-Minio} to prove a covering result that is similar to \cite[Proposition~5.1]{BDF}, but offers tighter control over the location of the sets being covered. The proof is based on that given in \cite{BDF}.

\begin{lem}\label{lem:Qj}
Let $f\colon \R^d\to\R^d$ be $K$-quasiregular of transcendental type, let ${q=q(d,K)}$ be Rickman's constant and let $U_1,\ldots,U_q$ be bounded sets in $\R^d$ with pairwise disjoint closures. Then, for all sufficiently large $r$ and each ${j\in\{1,\ldots,2q\}}$, we have
\eqn \label{Qj*}
f(Q_j(r)) \supset Q_i(M(r,f)), \quad \mbox{for some } i\in\{1,\ldots,2q\},
\eqnend
and
\eqn  \label{Qj**}
f(Q_j(r)) \supset M(r,f)U_l, \quad \mbox{for some } l\in\{1,\ldots,q\}.
\eqnend
\end{lem}

\begin{proof}
Take $j\in\{1,\ldots,2q\}$ and, for $r>0$, define $h_r\colon Q_j(1)\to\R^d$ by
\[ h_r(x) = \frac{f(rx)}{M(r,f)}. \]
There are points $x\in Q_j(1)$ with $|x|\le\frac12$ (for example, $x=v_j/2$) and for such points
\eqn \label{Qj***}
 |h_r(x)| \le \frac{M(r/2,f)}{M(r,f)} \to 0 \ \mbox{ as } r\to\infty
\eqnend
by Lemma~\ref{lem:M(Ar)}. Next choose $y_r'\in\R^d$ with $|y_r'|=\left(j+\tfrac14\right)r$ and $|f(y_r')| = M\!\left(\left(j+\tfrac14\right)r,f\right)$. Setting $y_r=y_r'/r$ gives that $y_r\in Q_j(1)$ with $|y_r|=j+\tfrac14$ and
\[ |h_r(y_r)| = \frac{|f(y_r')|}{M(r,f)} = \frac{M\!\left(\left(j+\tfrac14\right)r,f\right)}{M(r,f)}. \]
Thus $|h_r(y_r)|\to\infty$ as $r\to\infty$ by Lemma~\ref{lem:M(Ar)}. Together with \eqref{Qj***}, this shows that the family $\{h_r:r>0\}$ is not normal on $Q_j(1)$. In fact, for any sequence $r_k\to\infty$, the family $\{h_{r_k}\}$ is not normal. Therefore, it follows from Lemma~\ref{lem:Q-Minio} that, for all large $r$, there exists $i\in\{1,\ldots,2q\}$ such that $h_r(Q_j(1))\supset Q_i(1)$. Recalling the definition of $h_r$ now yields \eqref{Qj*}.

Since $U_1,\ldots,U_q$ are bounded sets with disjoint closures and the functions $h_r$ are quasiregular, it follows similarly from Lemma~\ref{lem:Minio} that, for all large~$r$, there is $l\in\{1,\ldots,q\}$ such that $h_r(Q_j(1))\supset U_l$. This in turn implies \eqref{Qj**}.
\end{proof}

The next lemma concerns quasiregular functions that do not have the pits effect as defined in Subsection~\ref{sect:pits}. A large part of the proof is closely based on \cite[\S 3]{BN}.

\begin{lem}\label{lem:Uj}
Let $f\colon \R^d\to\R^d$ be a $K$-quasiregular map of transcendental type that does not have the pits effect and let $q=q(d,K)$ be Rickman's constant. Then there exist a sequence $(R_k)$ tending to infinity and bounded domains $U_1,\ldots,U_q$ with pairwise disjoint closures in $\{x\in\R^d : |x|\ge 1/2\}$ with the following properties:

For every $k\in\N$ and $j\in\{1,\ldots,q\}$,
\begin{enumerate}
\item[(i)] there is some $l\in\{1,\ldots,q\}$ such that \[ f(R_kU_j) \supset R_kU_l; \]
\item[(ii)] for every $s\in [1,M(R_k,f)]$, there is some $i\in\{1,\ldots,2q\}$ such that
\[ f(R_kU_j) \supset Q_i(s); \]
\item[(iii)] $R_k\overline{U_j} \cap J(f) \ne\emptyset$.
\end{enumerate}
\end{lem}

\begin{proof}
Let $p$ be an integer satisfying $p>K_I(f)+q$. Since $f$ does not have the pits effect, there exist $c>1$ and $\varepsilon>0$ and a sequence $R_k\to\infty$ such that
\[ \{x\in\R^d : R_k\le|x|\le cR_k,\ |f(x)|\le 1\} \]
cannot be covered by  $p$ balls of radius $\varepsilon R_k$. Equivalently, by denoting \[{A=\{x\in\R^d : 1\le|x|\le c\},}\] the set $\{x\in A : |f(R_k x)|\le 1\}$ cannot be covered by $p$ balls of radius $\varepsilon$. Thus there exist $x_1^k, \ldots, x_p^k \in A$ satisfying $|x_i^k - x_j^k|\ge\varepsilon$ for $i\ne j$ such that ${|f(R_kx_j^k)|\le 1}$ for $j\in\{1,\ldots,p\}$.

Passing to a subsequence if necessary, we may assume that the sequences $(x_j^k)_{k\in\N}$ converge, say $x_j^k\to x_j$ as $k\to\infty$. Then $|x_i - x_j|\ge\varepsilon$ for $i\ne j$. We fix $r_1,\ldots,r_p$ satisfying $1<r_1<r_2<\ldots<r_p<c$ and choose $y_j^k$ such that $|y_j^k|=r_j$ and $|f(R_ky_j^k)|=M(R_kr_j,f)$. Again we may assume that the sequences $(y_j^k)_{k\in\N}$ converge, say $y_j^k\to y_j$ as $k\to\infty$. We may choose pairwise disjoint curves $\gamma_j$ in $A$ which connect $x_j$ with $y_j$ and small neighbourhoods $U_j$ of the curves $\gamma_j$ such that the closures $\overline{U_j}$ are pairwise disjoint and contained in $\{x\in\R^d : |x|\ge1/2\}$.

Consider the sequence of functions $g_k\colon\R^d\to\R^d$ given by
\[ g_k(x) = \frac{f(R_kx)}{R_k}. \]
Note that
\[ |g_k(x_j^k)| = \frac{|f(R_kx_j^k)|}{R_k} \le \frac{1}{R_k}, \]
while
\[ |g_k(y_j^k)| = \frac{|f(R_ky_j^k)|}{R_k} = \frac{M(R_kr_j,f)}{R_k} \to \infty \]
as $k\to\infty$ by using Lemma~\ref{lem:M(Ar)}. Since $x_j^k\to x_j \in U_j$ and $y_j^k\to y_j \in U_j$ as $k\to\infty$, it follows that no subsequence of $(g_k)$ is normal on any of the domains $U_j$. Thus we may deduce from Lemma~\ref{lem:Minio} that, if $k$ is large  and if $j\in\{1,\ldots, p\}$, then $g_k(U_j) \supset U_l$ for at least $p-q+1$ values of $l\in\{1,\ldots,p\}$. After discarding initial terms of $(R_k)$ and relabelling, we obtain that if $k\in\N$ and $j\in\{1,\ldots,p\}$, then $f(R_kU_j)\supset R_kU_l$ for at least $p-q+1$ values of $l\in\{1,\ldots,p\}$. This immediately implies that (i) holds and applying Lemma~\ref{lem:Julia2} to the domains $R_kU_1,\ldots,R_kU_p$ yields (iii).

As we are free to discard further initial terms of $(R_k)$, it will suffice to prove that (ii) holds for all large integers $k$. Suppose that this is not the case. Then there must exist a subsequence $(R_{k_n})$, an integer $j\in\{1,\ldots,q\}$ and values ${s_n\in[1,M(R_{k_n},f)]}$ such that $f(R_{k_n}U_j)$ omits a value in each of ${Q_1(s_n),\ldots, Q_{2q}(s_n)}$. Let $F_n\colon U_j\to\R^d$ be defined by
\[ F_n(x) = \frac{f(R_{k_n}x)}{s_n}. \]
Then every function $F_n$ omits a value in each of the sets $Q_1(1),\ldots, Q_{2q}(1)$. Therefore the family $\{F_n:n\in\N\}$ is normal on $U_j$ by Lemma~\ref{lem:Q-Minio}.

On the other hand, we have that
\[ |F_n(x_j^{k_n})| = \frac{|f(R_{k_n}x_j^{k_n})|}{s_n} \le \frac{1}{s_n} \le 1. \]
Also
\[ |F_n(y_j^{k_n})| = \frac{|f(R_{k_n}y_j^{k_n})|}{s_n} = \frac{M(R_{k_n}r_j,f)}{s_n} \ge \frac{M(R_{k_n}r_j,f)}{M(R_{k_n},f)}, \]
and so $|F_n(y_j^{k_n})|\to\infty$ as $n\to\infty$ by Lemma~\ref{lem:M(Ar)} because $r_j>1$ and $R_{k_n}\to\infty$. Since $x_j^{k_n}\to x_j \in U_j$ and $y_j^{k_n}\to y_j \in U_j$ as $n\to\infty$, this stands in contradiction to the normality of the family $\{F_n:n\in\N\}$.
\end{proof}

We are now ready to prove Theorem~\ref{thm1} under the assumption that $f$ is a transcendental type quasiregular mapping that does not have the pits effect. We shall do this by using the results established above to choose sets $A_\nu^{(j)}$ that satisfy all the hypotheses of the holding-up Lemma~\ref{lem:holdup}. We begin by taking a sequence $(R_k)$ and domains $U_1,\ldots,U_q$ as in Lemma~\ref{lem:Uj}. We may assume that $R_{k+1}>M(R_k,f)$
and that $R_1\ge1$ is large enough that $M^L(R_1,f)\to\infty$ as $L\to\infty$.
This means that for each $k\in\N$ there is a least integer $L_k\ge2$ such that $M^{L_k}(R_k,f)\ge R_{k+1}$. Define the sequence $(\nu_k)$ inductively by $\nu_1=1$ and $\nu_{k+1}=\nu_k+L_k$ for $k\ge1$. Then let
\[ A_{\nu_k}^{(j)} = R_kU_j \]
for $j\in\{1,\ldots,q\}$. (Additionally, for $j\in\{q+1,\ldots,2q\}$ we set $A_{\nu_k}^{(j)} = A_{\nu_k}^{(1)}$, say, simply to keep the number of sets consistent at a later stage.) Parts~(i) and~(iii) of Lemma~\ref{lem:Uj} immediately show that these sets  satisfy conditions \eqref{Hc} and \eqref{Hd} of Lemma~\ref{lem:holdup} with $p=2q$. It remains to choose the rest of the sets $A_\nu^{(j)}$ and then to verify \eqref{Ha} and \eqref{Hb}.

Note that the iterated maximum modulus function $M^{L_k-1}(r,f)$ is continuous in $r$ and that
\[ M^{L_k-1}(R_k,f) \le R_{k+1} \le M^{L_k}(R_k, f) \]
by our choice of $L_k$. Thus, by the intermediate value theorem, there exists $S_k\in[R_k, M(R_k,f)]$ such that $M^{L_k-1}(S_k,f)=R_{k+1}$. For integers $\nu$ in the range $\nu_k<\nu<\nu_{k+1}$, we now define
\[ A_\nu^{(j)} = Q_j(M^{\nu-\nu_k-1}(S_k,f)) \]
for $j\in\{1,\ldots,2q\}$. Next, we verify that the covering condition  \eqref{Hb} is satisfied by considering three cases.
\begin{enumerate}
	\item If $\nu=\nu_k$ for some $k$, then for each $j\in\{1,\ldots,q\}$ we have
	\[ f(A_\nu^{(j)}) = f(R_kU_j) \supset Q_i(S_k) = A_{\nu_k+1}^{(i)} \]
	for some $i\in\{1,\ldots,2q\}$ by Lemma~\ref{lem:Uj}(ii). (And for $q<j\le2q$ recall that we took $A_{\nu_k}^{(j)}=A_{\nu_k}^{(1)}$.)
	\item If $\nu_k<\nu<\nu_{k+1}-1$ for some $k$, then for each $j\in\{1,\ldots,2q\}$ we have
	\[ f(A_\nu^{(j)}) = f(Q_j(M^{\nu-\nu_k-1}(S_k,f))) \supset Q_i(M^{\nu-\nu_k}(S_k,f)) = A_{\nu+1}^{(i)} \]
	for some $i\in\{1,\ldots,2q\}$ by \eqref{Qj*} of Lemma~\ref{lem:Qj}.
	\item If $\nu=\nu_{k+1}-1$ for some $k$, then
	\[ A_\nu^{(j)} = Q_j(M^{\nu_{k+1}-\nu_k-2}(S_k,f)) = Q_j(M^{L_k-2}(S_k,f)) \]
	and so, for each $j\in\{1,\ldots,2q\}$ we have
	\[ f(A_\nu^{(j)}) \supset M^{L_k-1}(S_k,f)U_l = R_{k+1}U_l = A_{\nu_{k+1}}^{(l)} = A_{\nu+1}^{(l)}  \]
	for some $l\in\{1,\ldots,q\}$ by using \eqref{Qj**} of Lemma~\ref{lem:Qj} and recalling the definition of $S_k$.
\end{enumerate}
We have therefore shown that \eqref{Hb} holds for all $\nu\in\N$. Finally, we must prove that the sets $A_\nu=\bigcup_{j=1}^{2q} A_\nu^{(j)}$ tend to infinity in the sense of \eqref{Ha}. To see this, we note first that
\[ \inf \{|x| : x\in A_{\nu_k}\} = \inf \left\{|x| : x\in \bigcup_{j=1}^{2q} R_kU_j \right\} \ge \frac{R_k}{2} \]
since $U_j\subset \{x\in\R^d : |x|\ge1/2\}$ (see Lemma~\ref{lem:Uj}). Recalling the definition of $Q_j(r)$ reveals that $\inf \{|x| : x\in Q_j(r)\}=r/4$ and hence, for $\nu_k<\nu<\nu_{k+1}$,
\[ \inf\{|x| : x\in A_{\nu}\} = \frac{M^{\nu-\nu_k-1}(S_k,f)}{4} \ge \frac{S_k}{4} \ge \frac{R_k}{4}. \]
Therefore \eqref{Ha} is satisfied because $(R_k)$ tends to infinity.

An application of Lemma~\ref{lem:holdup} now concludes the proof of Theorem~\ref{thm1} for functions without the pits effect.

\section{Proof of Theorem~\ref{thm2}}\label{sect:thm2}

Theorem~\ref{thm2} will be proved by using a modified version of the covering result Lemma~\ref{lem:Qj}, in which the sets $Q_j(r)$ will be replaced by similar sets that lie in increasingly thin spherical shells. By finding a point $x$ whose forward orbit runs through such sets, we will achieve that $|f^n(x)|=(1+o(1))M^n(R,f)$ as $n\to\infty$.

We begin with a definition similar to the one at the start of Section~\ref{sect:no-pits}. Let $q$ be an integer and fix $2q$ distinct unit vectors $v_1,\ldots, v_{2q}$ and a small angle $\theta >0$. For $j\in\{1,\ldots,2q\}$ and $k\in\N$, define the truncated cone
\[ C_{j,k} = \left\{x\in\R^d : 1-\frac{1}{k+1} < |x| < 1+\frac{1}{k}, \ \frac{v_j\cdot x}{|x|}>\cos\theta \right\}. \]
Since $\theta$ is small, the $2q$ sets $C_{1,k},\ldots,C_{2q,k}$ have disjoint closures for each fixed $k\in\N$.
The reader only interested in the case of transcendental entire functions on the complex plane may simply take $q=2$, any angle $\theta\in (0,\pi/4)$ and ${C_{j,k} = \left\{z\in\C : 1-\frac{1}{k+1} < |z| < 1+\frac{1}{k}, \ |\arg(z)-j\pi/2|<\theta \right\}}$.
Now define, for $r>0$,
\[ Q_{j,k}(r) = A\left(\left(1+\frac{j-1/2}{2qk}\right)r, \left(1+\frac{j}{2qk}\right)r\right) \cup rC_{j,k} \]
and observe that
\eqn \label{QinA}
Q_{j,k}(r) \subset A\left((1-\tfrac{1}{k+1})r, (1+\tfrac{1}{k})r\right),
\eqnend
so that for large values of $k$ the modulus of any point in $Q_{j,k}(r)$ is approximately~$r$. The following covering result is essentially just two copies of \eqref{Qj*} of Lemma~\ref{lem:Qj}.

\begin{lem}\label{lem:Qjk}
Let $f\colon \R^d\to\R^d$ be $K$-quasiregular of transcendental type and let $q=q(d,K)$ be Rickman's constant. Then for each $k\in\N$ there exists $r_k>0$ such that, for every $r\ge r_k$ and $j\in\{1,\ldots,2q\}$,
\eqn \label{Qalpha}
f(Q_{j,k}(r)) \supset Q_{i,k}(M(r,f)), \quad \mbox{for some } i\in\{1,\ldots,2q\},
\eqnend
and
\eqn  \label{Qbeta}
f(Q_{j,k}(r)) \supset  Q_{i',k+1}(M(r,f)), \quad \mbox{for some } i'\in\{1,\ldots,2q\}.
\eqnend
\end{lem}

\begin{proof}
Note first that Lemma~\ref{lem:Q-Minio} remains valid if the sets $Q_1(1),\ldots,Q_{2q}(1)$ are replaced by $Q_{1,k}(1),\ldots,Q_{2q,k}(1)$ for any $k\in\N$.

We now fix some $k\in\N$ and $j\in\{1,\ldots,2q\}$ and follow the proof of Lemma~\ref{lem:Qj} to show that \eqref{Qalpha} and \eqref{Qbeta} hold for all sufficiently large $r$. In detail, we again define $h_r\colon Q_{j,k}(1)\to\R^d$ by
\[ h_r(x) = \frac{f(rx)}{M(r,f)} \]
and note that we have, for example, $a=(1-\frac{1}{4k})v_j\in Q_{j,k}(1)$ and
\eqn \label{Qgamma}
 |h_r(a)| \le \frac{M((1-\frac{1}{4k})r,f)}{M(r,f)} \to 0 \ \mbox{ as } r\to\infty
\eqnend
by Lemma~\ref{lem:M(Ar)}. We can also find $y_r\in Q_{j,k}(1)$ with $|y_r|=1+\frac{j-1/4}{2qk}$ and $|f(ry_r)|=M(r|y_r|, f)$. Then, since $|y_r|\ge 1+\frac{3}{8qk}>1$,
\[ |h_r(y_r)| =  \frac{M(r|y_r|,f)}{M(r,f)}\to \infty \quad\mbox{as } r\to\infty \]
by Lemma~\ref{lem:M(Ar)}. Together with \eqref{Qgamma}, this shows that the family $\{h_r:r>0\}$ is not normal on $Q_{j,k}(1)$. In fact, for any real sequence $s_n\to\infty$, the family $\{h_{s_n}\}$ is not normal. As remarked above, Lemma~\ref{lem:Q-Minio} holds for the sets $Q_{1,k}(1),\ldots,Q_{2q,k}(1)$ and also for the sets $Q_{1,k+1}(1),\ldots,Q_{2q,k+1}(1)$. We therefore deduce that, for all large $r$, there exist $i,i'\in\{1,\ldots,2q\}$ such that $h_r(Q_{j,k}(1))\supset Q_{i,k}(1)$ and $h_r(Q_{j,k}(1))\supset Q_{i',k+1}(1)$, from which \eqref{Qalpha} and \eqref{Qbeta} follow.
\end{proof}

\begin{proof}[Proof of Theorem~\ref{thm2}]
Let $f$ be quasiregular of transcendental type and let $q$ be Rickman's constant. For clarity we will assume that $c=2$; a similar argument can be used in the general case. Take $r_1,r_2,\ldots$ as in Lemma~\ref{lem:Qjk}. By increasing $r_1$ if necessary, we may assume that $M(r,f)>r$ for all $r\ge r_1$. Let $R\ge r_1$.

We aim to choose a sequence of sets $E_n$ of the form $Q_{j,k_n}(M^n(R,f))$ such that $f(E_n)\supset E_{n+1}$ and
\eqn \label{Mn>rkn}
M^n(R,f) \ge r_{k_n}.
\eqnend
We begin by setting $E_0=Q_{1,1}(R)$, which satisfies \eqref{Mn>rkn} because $M^0(R,f)=R\ge r_1 = r_{k_0}$. We proceed by induction. Suppose that $E_n= Q_{j,k_n}(M^n(R,f))$ has been chosen such that \eqref{Mn>rkn} holds. This means that Lemma~\ref{lem:Qjk} may be applied to the set $E_n$.
\begin{itemize}
\item[\emph{Case 1}] If $M^{n+1}(R,f)\ge r_{k_n+1}$, then take $E_{n+1}=Q_{i',k_n+1}(M^{n+1}(R,f))$, where $i'\in\{1,\ldots,2q\}$ is chosen using \eqref{Qbeta} to give that $f(E_n)\supset E_{n+1}$. In this case, $k_{n+1}=k_n+1$ and \eqref{Mn>rkn} holds for $n+1$ because $M^{n+1}(R,f) \ge r_{k_n+1}=r_{k_{n+1}}$.
\item[\emph{Case 2}] Otherwise, we take $E_{n+1}=Q_{i,k_n}(M^{n+1}(R,f))$, where $i\in\{1,\ldots,2q\}$ is chosen using \eqref{Qalpha} to give that $f(E_n)\supset E_{n+1}$. In this case, $k_{n+1}=k_n$ and \eqref{Mn>rkn} holds for $n+1$ because $M^{n+1}(R,f) \ge M^n(R,f) \ge r_{k_n} = r_{k_{n+1}}$.
\end{itemize}
Observe that Case 1 must occur infinitely often because $M^n(R,f)\to\infty$ as $n\to\infty$. Thus $k_n\to\infty$ as $n\to\infty$.

We now apply Lemma~\ref{lem:En} to obtain a point $x$ such that $f^n(x)\in \overline{E_n}$ for all $n\ge 0$. In particular, $x\in \overline{E_0}=\overline{Q_{1,1}(R)}$ and so $R/2\le |x| \le 2R$ by~\eqref{QinA}. Moreover, \eqref{QinA} also yields
\[ \left(1-\frac{1}{k_n+1}\right)M^n(R,f) \le |f^n(x)| \le \left(1+\frac{1}{k_n}\right)M^n(R,f). \]
Since $k_n\to\infty$, this shows that the iterates $f^n(x)$ escape to infinity at the precise rate asserted in the statement of the theorem.
\end{proof}

\section{Proof of Theorem~\ref{thm3}}\label{sect:thm3}

We begin the proof of Theorem~\ref{thm3} by using the argument from the end of Section~5 of \cite{RSSlow} to prove quickly that statement~(a) implies statement~(c). To do this, suppose that (c) does not hold. Then we can find a sequence of spherical shells $A(r_n,R_n)$, where $0<r_n<R_n$, such that $r_n\to\infty$ and $R_n/r_n\to\infty$ as $n\to\infty$ and
\[ |f(x)| > M(|x|,f)^\frac{1}{2} \quad \mbox{ for } x\in A(r_n,R_n). \]
Note that Lemma~\ref{lem:M(Ar)} implies that $M(r,f)^\frac{1}{2}/r\to\infty$ as $r\to\infty$. Hence, it is not possible to satisfy (a) for any positive sequence $a_n$ that has a subsequence $a_{n_j}$ such that $a_{n_j}=a_{n_j+1}=r_j$.

The proof that (b) implies (a) is given in the next subsection. Finally, in Subsection~\ref{sect:c-implies-a}, Harnack's inequality will be used to prove that (c) implies~(b).

\subsection{Proof that (b) implies (a)}

We use the following covering lemma (cf. \cite[Lemma~8]{RSSlow}), the proof of which is similar to that of Lemmas~\ref{lem:Qj} and \ref{lem:Uj} above. We adopt the notation $\overline{A}(r,s)=\{x\in\R^d : r\le |x| \le s\}$.

\begin{lem}\label{lem:shells}
Let $f\colon \R^d\to\R^d$ be a $K$-quasiregular map of transcendental type and let $p\in\N$ be such that $p\ge q$, where $q=q(d,K)$ is Rickman's constant. Suppose that $c,\lambda >0$ and $L>1$. Then there exists $R_0\ge 1$ with the following property: If $r\ge R_0$, $\alpha\ge 1$,
\eqn \label{eqn:1<R<lambdaM}
 1 \le R \le \lambda M(r,f) \quad \mbox{ and } \quad \min\{|f(x)| : x\in \overline{A}(2\alpha r, 2L\alpha r) \} \le c,
\eqnend
then
\[ f(A(\alpha r, 4L\alpha r)) \supset A((8L)^{i-1}R, ((8L)^i/2)R) \]
for at least $p-q+1$ values of $i\in\{1,\ldots,p\}$.
\end{lem}

\begin{proof}
Suppose that no such $R_0$ exists. Then there exist real sequences $(r_k)$, $(\alpha_k)$ and $(R_k)$ such that $r_k\to\infty$ as $k\to\infty$ and, for all $k$,
\begin{itemize}
\item $\alpha_k \ge 1$;
\item \eqref{eqn:1<R<lambdaM} is satisfied with $R=R_k$, $r=r_k$ and $\alpha=\alpha_k$;
\item but $f(A(\alpha_k r_k, 4L\alpha_k r_k))$ does not contain $A((8L)^{i-1}R_k, ((8L)^i/2)R_k)$ for at least $q$ choices of $i\in\{1,\ldots,p\}$.
\end{itemize}
Let $g_k\colon A(1,4L)\to\R^d$ be the $K$-quasiregular map defined by
\[ g_k(x) = \frac{f(\alpha_kr_k x)}{R_k}. \]
Then $g_k$ omits a point in $A((8L)^{i-1}, (8L)^i/2)$ for at least $q$ values of $i\in\{1,\ldots,p\}$. It now follows from Lemma~\ref{lem:Minio} that the family $\{g_k:k\in\N\}$ is normal on $A(1,4L)$.

On the other hand, we see that
\eqn \label{eqn:min(g_k)}
\min \{|g_k(x)| : x\in \overline{A}(2,2L) \} \le c/R_k \le c
\eqnend
while also
\eqn \label{eqn:M(2,g_k)}
M(2,g_k) = \frac{M(2\alpha_k r_k, f)}{R_k} \ge \frac{M(2r_k,f)}{\lambda M(r_k,f)}.
\eqnend
By Lemma~\ref{lem:M(Ar)}, this last term tends to infinity as $k\to\infty$. Therefore, \eqref{eqn:min(g_k)} and \eqref{eqn:M(2,g_k)} together contradict the normality of the family $\{g_k\}$.
\end{proof}

We now assume that $f$ is a transcendental type $K$-quasiregular map such that statement~(b) in Theorem~\ref{thm3} holds for some $L$, $c$ and $(x_n)$. We aim to show that statement~(a) also holds.

Let $(a_n)$ be a positive sequence as in (a). Fix an integer $p>K_I(f)+q(K,d)$ and take $\lambda>0$ and $N\in\N$ such that, for all $n\ge N$,
\[ 1 \le a_{n+1} \le \lambda M(a_n, f) \quad \mbox{ and } \quad a_n \ge \max\{|x_1|, R_0\}, \]
where $R_0$ is given by Lemma~\ref{lem:shells}. Write $\tau=8L$. Now, whenever $\alpha\ge 1$ and $n\ge N$, the set $\overline{A}(2\alpha a_n, 2L\alpha a_n)$ must contain one of the points $x_m$ and thus, by applying Lemma~\ref{lem:shells} with $r=a_n$ and $R=a_{n+1}$, we find that there exists some $i\in\{1,\ldots,p\}$ such that
\[ f(A(\alpha a_n, (\tau/2)\alpha a_n)) \supset A( \tau^{i-1}a_{n+1}, (\tau^i/2)a_{n+1}). \]
It follows that we may choose a sequence of sets $(E_n)_{n\ge N}$ of the form
\[ E_n = A(\beta_n a_n, (\tau/2)\beta_n a_n), \quad \mbox{ where } \beta_n\in\{1,\tau,\ldots,\tau^{p-1}\}, \]
such that $f(E_n)\supset E_{n+1}$. Moreover, Lemma~\ref{lem:shells} also gives that, for $n\ge N$ and each $j\in\{1,\ldots,p\}$,
\[ f(A(\tau^{j-1}a_n, (\tau^j/2)a_n)) \supset A(\tau^{i-1}a_n, (\tau^i/2)a_n), \]
for at least $p-q+1$ values $i\in\{1,\ldots,p\}$. Hence, by Lemma~\ref{lem:Julia2}, we have in particular that $\overline{E_n}\cap J(f) \ne \emptyset$ for $n\ge N$. Therefore, we may apply Lemma~\ref{lem:En} to obtain a point $\zeta_N\in J(f)$ such that $f^{n-N}(\zeta_N)\in\overline{E_n}$ for $n\ge N$. It follows that, for $n\ge N$,
\[ a_n \le |f^{n-N}(\zeta_N)| \le (\tau^p/2) a_n. \]

Without loss of generality, we may assume that $\zeta_{N}\notin\mathcal{E}(f)$ (otherwise we increase~$N$). Thus, by applying Lemma~\ref{lem:Rickman} finitely many times, we can find $\zeta$ such that $f^{N}(\zeta)=\zeta_{N}$ and (a) is satisfied for some suitable choice of $C$.
Since $\zeta_N\in J(f)$, we also have $\zeta\in J(f)$ by the complete invariance of the Julia set.

\subsection{Proof that (c) implies (b)}\label{sect:c-implies-a}

We will use the following quasiregular version of Harnack's inequality; see \cite[Theorem VI.7.4, Corollary VI.2.8]{Rickman}.

\begin{lem}\label{lem:Harnack}
If $g\colon B(w,R)\to \R^d$ is $K$-quasiregular and $|g(x)|>1$ for all $x\in B(w,R)$, then, for $0<\rho < R$,
\[ \sup_{x\in B(w,\rho)} \log |g(x)| \le a \inf_{x\in B(w,\rho)} \log |g(x)|, \]
where $a=\exp(C_{d,K}(\log(R/\rho))^{-1})$ and the constant $C_{d,K}$ depends only on $d$ and~$K$.
\end{lem}

In the complex transcendental entire case, the result we seek can be proved by applying Harnack's inequality to the composition $g=f \circ \exp$; see \cite[Lemma~5(a)]{RSSlow} and \cite[Lemma~2]{Hinkkanen}. In the quasiregular setting, a class of maps introduced by Zorich provide a suitable analogue of the complex exponential function; see, for example, \cite[\S6.5.4]{Iwaniec01} or \cite[\S I.3.3]{Rickman}. For our purposes, it is enough to know that there exists a quasiregular Zorich map $Z\colon \R^d \to \R^d\setminus\{0\}$ that is periodic in the first $d-1$ co-ordinates, and that is a surjection from $[-1,1)^{d-1}\times\{t\}$ to $\{x\in\R^d: |x|=e^t\}$ for all $t\in\R$.

Assume now that the function $f$, sequence $(x_n')$ and constants $L'$ and $s$ satisfy statement~(c) of Theorem~\ref{thm3}.

Let $g=f \circ Z$, where $Z$ is the Zorich map mentioned above. Then $g$ is a $K'$-quasiregular map of $\R^d$ for some $K'\ge 1$. Choose $\beta>0$ sufficiently large that $\log \beta > 2\sqrt{d-1}$ and also that the Harnack constant~$a$ for the function~$g$ with radii $R=\log\beta$ and $\rho=2\sqrt{d-1}$ satisfies
\eqn \label{eqn:choice_a}
a= \exp\left(C_{d,K'}\left(\log\left(\frac{\log\beta}{2\sqrt{d-1}}\right)\right)^{-1}\right) < \frac1s.
\eqnend

We now claim that, for all $n\in\N$,
\eqn \label{eqn:exists_xn}
\mbox{there exists } x_n\in\R^d \mbox{ such that } \frac1\beta < \frac{|x_n|}{|x'_n|} < \beta \mbox{ and } |f(x_n)|\le 1.
\eqnend
Observe that statement (b) follows from (c) and this claim, by taking $L=\beta^2 L'$ and $c=1$.

In order to prove the claim, we suppose that \eqref{eqn:exists_xn} fails for some $n\in\N$. Writing $r=|x_n'|$, this means that $|f(x)|>1$ for all $x\in A(r/\beta,\beta r)$. The Zorich function $Z$ maps the region
\[ S=\R^{d-1}\times (\log(r/\beta), \log(\beta r)) \]
into $A(r/\beta, \beta r)$ and hence $|g(x)|>1$ for all $x\in S$.

Let $y\in\R^d$ be a maximum modulus point of $f$ with $|y|=r$; that is,
\eqn \label{eqn:y-maxmod}
|f(y)|= M(r,f) = M(|x'_n|, f).
\eqnend
By the property of $Z$  mentioned above, there exist points
\eqn \label{eqn:v,w}
v,w \in [-1,1)^{d-1}\times\{\log r\} \mbox{ with } Z(v)=x'_n \mbox{ and } Z(w)=y.
\eqnend
It follows that
\[ v\in B(w, 2\sqrt{d-1}) \subset B(w, \log\beta) \subset S. \]
Hence Harnack's inequality (Lemma~\ref{lem:Harnack}) now shows that $\log|g(w)| \le a \log|g(v)|$. Using \eqref{eqn:choice_a}, \eqref{eqn:y-maxmod} and \eqref{eqn:v,w}, this gives that $\log M(|x'_n|,f) < \frac1s\log|f(x_n')|$,  in contradiction to statement~(c).
This proves the claim that \eqref{eqn:exists_xn} holds for all $n\in\N$, and therefore completes the proof that (c) implies~(b).

\end{document}